\documentclass[11pt, reqno]{amsart}
\usepackage{amssymb,latexsym}
\usepackage{enumerate}
\usepackage{hyperref}
\usepackage{comment}
\usepackage[margin=1in]{geometry}
\usepackage{etoolbox}

\newcommand\blfootnote[1]{%
  \begingroup
  \renewcommand\thefootnote{}\footnote{#1}%
  \addtocounter{footnote}{-1}%
  \endgroup
}

\makeatletter \@namedef{subjclassname@2010}{%
  \textup{2010} Mathematics Subject Classification}
\makeatother

\newtheorem{Theorem}{Theorem}[section]

\newtheorem{Lemma}[Theorem]{Lemma}
\newtheorem{Corollary}[Theorem]{Corollary}

\newtheorem*{Remark}{Remark} 
\newtheorem{Conjecture}{Conjecture}

\newtheorem{Claim}[Theorem]{Claim}

\newcommand{\CC}[0]{\mathbb C}	
	\newcommand{\QQ}[0]{\mathbb Q}
	\newcommand{\RR}[0]{\mathbb R}

	\newcommand{\ZZ}[0]{\mathbb Z}

	\newcommand{\cP}[0]{\mathcal P}

\newcommand{\cJ}[0]{\mathcal J}

\newcommand{\dd}[0]{\textbf{\textit{d}}}

\newcommand{\jj}[0]{\textbf{\textit{j}}}

\renewcommand{\bar}[1]{\overline{#1}}

\newcommand{\eps}[0]{\varepsilon}

\newcommand{\lr}[1]{\left(#1\right)}

\newcommand\be{\begin{eqnarray*}}
\newcommand\ee{\end{eqnarray*}}
\newcommand\beq{\begin{equation}}
\newcommand\eeq{\end{equation}}

\newcommand\ben{\begin{eqnarray}}
\newcommand\een{\end{eqnarray}}

\begin{document}

% %%%% To ease editing, for IMPAN journals add:

\baselineskip=17pt

% %%%%%%%%%%

% % In the running head, replace first names by initials % and give an
% abbreviation of the title.

\title{On iterated product sets with shifts II}

\author[B. Hanson]{Brandon Hanson} \address{University of Georgia\\
Athens, GA, USA}
\email{brandon.w.hanson@gmail.com}
\author[O. Roche-Newton]{Oliver Roche-Newton} \address{Johann Radon Institute for Computational and Applied Mathematics\\
Linz, Austria}
\email{o.rochenewton@gmail.com}
\author[D. Zhelezov]{Dmitrii Zhelezov} \address{Alfr\'{e}d R\'{e}nyi Institute of Mathematics \\ 
Hungarian Academy of Sciences, Budapest, Hungary }
\email{dzhelezov@gmail.com}
\date{}

\begin{abstract} The main result of this paper is the following: for all $b \in \mathbb Z$ there exists $k=k(b)$ such that
\[ \max \{ |A^{(k)}|, |(A+u)^{(k)}|  \}  \geq |A|^b, \]
for any finite $A \subset \mathbb Q$ and any non-zero $u \in \mathbb Q$. Here, $|A^{(k)}|$ denotes the $k$-fold product set $\{a_1\cdots a_k : a_1, \dots, a_k \in A \}$.

%In order to prove this result, we use the language of the $\Lambda$-constants from harmonic analysis, applied to Dirichlet polynomials.
Furthermore, our method of proof also gives the following $l_{\infty}$ sum-product estimate. For all $\gamma >0$ there exists a constant $C=C(\gamma)$ such that for any $A \subset \mathbb Q$ with $|AA| \leq K|A|$ and any $c_1,c_2 \in \mathbb Q \setminus \{0\}$, there are at most $K^C|A|^{\gamma}$ solutions to
\[ c_1x + c_2y =1 ,\,\,\,\,\,\,\, (x,y) \in A \times A.
\]
In particular, this result gives a strong bound when $K=|A|^{\epsilon}$, provided that $\epsilon >0$ is sufficiently small, and thus improves on previous bounds obtained via the Subspace Theorem.

In further applications we give a partial structure theorem for point sets which determine many incidences and prove that sum sets grow arbitrarily large by taking sufficiently many products.

We utilise a query-complexity analogue of the polynomial Freiman-Ruzsa conjecture, due to P\"{a}lv\"{o}lgyi and Zhelezov \cite{PZ}. This new tool replaces the role of the complicated setup of Bourgain and Chang \cite{BC}, which we had previously used. Furthermore, there is a better quantitative dependence between the parameters.

%\fxnote{I hope I did not write some nonsense here in the abstract, please read critically.}
%We define a variant of multiplicative energy which controls the simultaneous multiplicative energy of a finite set and a translate of that set.

\end{abstract}

\maketitle

\blfootnote{Mathematics Subject Classification (2010) - 11B30}
\blfootnote{Declaration of interests: none}

\section{Introduction}

\subsection{Background and statement of main results}

Let $A$ be a finite set of rational numbers and let $u\in \QQ$ be non-zero. In this article we wish to investigate the sizes of the $k$-fold product sets
\[A^{(k)}:=\{a_1\cdots a_k:a_1,\ldots,a_k\in A\}\]
and
\[(A+u)^{(k)} =\{(a_1+u)\cdots (a_k+u):a_1,\ldots,a_k\in A\}.\]
This is an instance of a sum-product problem. Recall that the Erd\H{o}s-Szemer\'{e}di \cite{ES} sum-product conjecture states that, for all $\epsilon >0$ there exists a constant $c(\epsilon)>0$ such that
\[ \max \{|A+A|, |AA| \} \geq c(\eps) |A|^{2-\eps} \]
holds for any $A \subset \mathbb Z$. Here $A+A:=\{a+b : a,b \in A \}$ is the \textit{sum set} of $A$, and $AA$ is another notation for $A^{(2)}$. Erd\H{o}s and Szemer\'{e}di also made the more general conjecture that for any finite $A \subset \mathbb Z$,
\[
\max \{|kA|,|A^k|\} \geq c(\epsilon)|A|^{k-\epsilon},
\]
where $kA:= \{a_1+ \dots +a_k : a_1,\dots, a_k \in A\}$ is the \textit{$k$-fold sum set}. Both of these conjectures are wide open, and it is natural to also consider them for the case when $A$ is a subset of $\mathbb R$ or indeed other fields. The case when $k=2$ has attracted the most interest. See, for example, \cite{KS}, \cite{KS2}, \cite{S}, \cite{TV} and the references contained therein for more background on the original Erd\H{o}s-Szemer\'{e}di sum-product problem.

Most relevant to our problem is the case of general (large) $k$. Little is known about the Erd\H{o}s-Szemer\'{e}di conjecture in this setting, with the exception of the remarkable series of work of Chang \cite{C} and Bourgain-Chang \cite{BC}. This culminated in the main theorem of \cite{BC}: for all $b \in \mathbb R$ there exists $ k = k(b) \in \mathbb Z$ such that
\begin{equation} \label{BCmain}
\max \{|kA|,|A^k|\} \geq |A|^b
\end{equation}
holds for any $A \subset \mathbb Q$. On the other hand, it appears that we are not close to proving such a strong result for $A \subset \mathbb R$.

In the same spirit as the Erd\H{o}s-Szemer\'{e}di conjecture, it is expected that an additive shift will destroy  multiplicative structure present in $A$. In particular, one expects that, for a non-zero $u$, at least one of $|A^{(k)}|$ or $|(A+u)^{(k)}|$ is large. The $k=2$ version of this problem was considered in \cite{GS} and \cite{JRN}. The main result of this paper is the following analogue of the Bourgain-Chang Theorem.
\begin{Theorem} \label{thm:mainmain}
For all $b \in \mathbb Z$, there exists $k=k(b)$ such that for any finite set $A \subset \mathbb Q$ and any non-zero rational $u$,
\[  \max \{ |A^k|, |(A+u)^k| \} \geq |A|^b . \]
\end{Theorem}

This paper is a sequel to \cite{HRNZ}, in which the main result was the following.
\begin{Theorem} \label{thm:usold} For any finite set $A \subset \mathbb Q$ with $|AA| \leq K|A|$, any non-zero $u \in \mathbb Q$ and any positive integer $k$,
$$| (A+u) ^{(k)}| \geq \frac{|A|^k}{(8k^4)^{kK}}. $$
\end{Theorem}
The proof of this result was based on an argument that Chang \cite{C} introduced to give similar bounds for the $k$-fold sum set of a set with small product set. Theorem \ref{thm:usold} is essentially optimal when $K$ is of the order $c\log|A|$, for a sufficiently small constant $c=c(k)$. However, the result becomes trivial when $K$ is larger, for example if $K=|A|^{\epsilon}$ and $\eps>0$. The bulk of this paper is devoted to proving the following theorem, which gives a near optimal bound for the size of $(A+u)^{(k)}$ when $K=|A|^{\eps}$,
for a sufficiently small but positive $\eps$.
\begin{Theorem} \label{thm:main1}
Given $0<\gamma < 1/2$, there exists a positive constant $C=C(\gamma, k)$ such that for any finite $A \subset \mathbb Q$ with $|AA|=K|A|$ and any non-zero rational $u$,
\[ | (A+u) ^{(k)}|  \geq \frac{|A|^{k(1-\gamma)-1}}{K^{Ck}}.\]
\end{Theorem}

In fact, we prove a more general version of Theorem \ref{thm:main1} in terms of certain weighted energies and so-called $\Lambda$-constants (see Theorem \ref{thm:lambda} for the general statement that implies Theorem \ref{thm:main1} - see sections \ref{sec:energy} and \ref{sec:lambda} for the relevant definitions of energy and $\Lambda$-constants). This more general result is what allows us to deduce Theorem \ref{thm:mainmain}.

\subsection{A subspace type theorem -- an $l_\infty$ sum-product estimate}
%\fxnote{Olly: I have rearranged this part of the introduction to try and emphasise that our main application is Theorem 1.7}

It appears that Theorem \ref{thm:mainmain}, as well as the forthcoming generalised form of Theorem \ref{thm:main1}, lead to some interesting new applications. To illustrate the strength of these sum-product results, we present three applications in this paper.

Our main application concerns a variant of the celebrated Subspace Theorem by Evertse, Schmidt and Schlikewei \cite{evertse2002linear} which, after quantitative improvements by Amoroso and Viada \cite{amoroso2009small}, reads as follows.

Suppose $a_1, \ldots,  a_k \in \mathbb{C}^*$, $\alpha_1,\ldots,\alpha_r \in \mathbb C^*$ and define
$$
\Gamma = \{\alpha_1^{z_1} \cdots \alpha_r^{z_r}, z_i \in \mathbb{Z} \},
$$
so $\Gamma$ is a free multiplicative group\footnote{The original theorem is formulated in a more general setting, namely for the division group of $\Gamma$, but we will stick to the current formulation for simplicity.} of rank $r$. Consider the equation
\begin{equation} \label{eq:subspace_eq}
a_1x_1 + a_2x_2 + \cdots + a_kx_k = 1 
\end{equation}
with $a_i \in \mathbb{C}^*$ viewed as fixed coefficients and $x_i \in \Gamma$ as variables. A solution $(x_1, \ldots, x_k)$ to (\ref{eq:subspace_eq}) is called \emph{nondegenerate} if
for any non-empty $J \subsetneq \{1, \ldots, k \}$
$$
	\sum_{i \in J} a_ix_i \neq 0.
$$

\begin{Theorem}[The Subspace Theorem, \cite{evertse2002linear} \cite{amoroso2009small} ] \label{thm:subspace}
The number $A(k, r)$ of nondegenerate  solutions to (\ref{eq:subspace_eq}) satisfies the bound
\begin{equation} \label{subspace_thm_ineq}
A(k, r) \leq {(8k)}^{4k^4(k + kr + 1)}.
\end{equation}
\end{Theorem}

%, or rather its two-variables predecessor %by Beukers and Schlickewei \cite{}. The %result reads as follows.

%\begin{Theorem} \label{thm:BeukersSchlik}
%Let $G$ be the $\mathbb{Q}$-closure of a %finitely generated subgroup of %$(\mathbb{C^*})^2$
%of rank $r$. Then the equation
%$$
%x + y = 1, (x, y) \in G,
%$$
%has not more than $2^{8r+8}$ solutions.
%\end{Theorem}

The Subspace Theorem dovetails nicely to the following version of the Freiman Lemma. 
%\fxnote{Olly: I changed Freiman-Ruzsa Lemma to Freiman Lemma to be consistent with what we write later. Please change it back if I have misunderstood.}
\begin{Theorem} \label{thm:FR-lemma}
Let $(G, \cdot)$ be a torsion-free abelian group and 
$A \subset G$ with $|AA| < K|A|$. Then $A$ is contained in a subgroup $G' < G$ of rank at most $K$.
\end{Theorem}

Now assume for simplicity that $A \subset \mathbb{Q}$ and $|AA| \leq K|A|$. 
 Let us call such sets (this definition generalizes of course to an arbitrary ambient group) $K$-\emph{almost subgroups} \footnote{One could've used a more general framework of $K$-\emph{approximate subgroups} introduced by Tao. We decided to introduce a simpler definition in order to avoid technicalities. However, in the abelian setting the definitions are essentially equivalent.}. %\fxnote{Olly: I'm not a big fan of this definition... but can live with it.}
 
% With such a definition a $K$-almost %subgroup has (as a finite set) %multiplicative rank at most $K$ by %Theorem~\ref{thm:FR-lemma}. 

We now show that it is natural to expect that the Subspace Theorem generalises to $K$-almost subgroups with $K$ taken as a proxy for the group rank. A straightforward corollary of Theorem~\ref{thm:FR-lemma} and Theorem~\ref{thm:subspace} is as follows. 

\begin{Corollary}[Subspace Theorem for $K$-almost subgroups] \label{corr:subspace_almost_subgroups}
  Let $A$ be a $K$-almost subgroup. Then the number $A(k, K)$ of non-degenerate solutions $(x_1, x_2, \ldots, x_k) \in A^k$ to
$$
c_1x_1 + c_2x_2 + \ldots + c_kx_k = 1
$$  
with fixed coefficients $c_i \in \mathbb{C^*}$ is bounded by 
$$ 
A(k, K) \leq {(8k)}^{4k^4(k + kK + 1)}.
$$
  
\end{Corollary}

Similarly to Theorem~\ref{BCmain}, the bound of Corollary~\ref{corr:subspace_almost_subgroups} becomes trivial when $A$ is large and $K$ is larger than $c\log |A|$ for some small $c > 0$. 

We conjecture that a much stronger polynomial bound holds.

\begin{Conjecture} \label{conj:Ksubspace}
 There is a constant $c(k)$ such that Corollary~\ref{corr:subspace_almost_subgroups} holds with the bound
 $$
  A(k, K) \leq K^{c(k)}.
 $$
\end{Conjecture}

We can support Conjecture~\ref{conj:Ksubspace} with a special case $k = 2$ and $A \subset \mathbb{Q}, c_i \in \mathbb{Q}$ and a somewhat weaker estimate, which we see as a proxy for the Beukers-Schlikewei Theorem \cite{beukers1996equation}.

\begin{Theorem}[Weak Beukers-Schlikewei for $K$-almost subgroups] \label{thm:BS_almost_subgroups}
 For any $\gamma > 0$ there is $C(\gamma) > 0$ such that for any $K$-almost subgroup $A \subset \mathbb{Q}$ and fixed non-zero $c_1, c_2 \in \mathbb{Q}$ the number $A(2, K)$ of solutions $(x_1, x_2) \in A^2$ to 
 $$
 c_1x_1 + c_2 x_2 = 1
 $$
 is bounded by 
 $$
 A(2, K) \leq |A|^\gamma K^C.
 $$

\end{Theorem}
%\fxnote{Olly: Added some $l_{\infty}$ spiel here.}
One can view Theorem \ref{thm:BS_almost_subgroups} as an $l_{\infty}$ version of the weak Erd\H{o}s-Szemer\'{e}di sum-product conjecture. The \textit{weak Erd\H{o}s-Szemer\'{e}di conjecture} is the statement that, if $|AA| \leq K|A|$ then $|A+A| \geq K^{-C}|A|^2$ for some positive absolute constant $C$. For $A \subset \mathbb Z$, this result was proved in \cite{BC}, but the conjecture remains open over the reals. 

A common approach to proving sum-product estimates is to attempt to show that, for a set $A$ with small product set, %, say $|AA| \leq K|A|$, 
the \textit{additive energy} of $A$, which is defined as the quantity
\[E_+(A):= |\{ (a,b,c,d) \in A^4 : a+b=c+d \}|, \]
is small. Indeed, this was the strategy implemented in \cite{C} and \cite{BC}, the latter of which showed\footnote{This is something of an over-simplification, as \cite{BC} in fact proved a much more general result which bounded the multi-fold additive energy with weights attached.} that, for all $\gamma >0$, there is a constant $C=C(\gamma)$ such that for any $A \subset \mathbb Q$ with $|AA| \leq K|A|$,
\begin{equation} 
E_+(A) \leq K^C|A|^{2+ \gamma}. 
\label{eq:weakES}
\end{equation}
Since there are at least $|A|^2$ trivial solutions when $\{a,b\}=\{c,d\}$, this bound is close to best possible. It then follows from a standard application of the Cauchy-Schwarz inequality that
\[|A+A| \geq \frac{|A|^{2-\gamma}}{K^C}. \]
Defining the representation function $r_{A+A}(c)=|\{(a_1,a_2) \in A \times A : a_1+a_2=c\}|$, it follows that
\[E_+(A)= \sum_x r_{A+A}(x)^2,\]
and so bounds for the additive energy can be viewed as $l_2$ estimates for this representation function.

Theorem \ref{thm:BS_almost_subgroups} gives the stronger $l_{\infty}$ estimate: it says that, if $|AA| \leq K|A|$ then $r_{A+A}(c) \leq K^C|A|^{\gamma}$ for all $c \neq 0$. This implies \eqref{eq:weakES}, and thus in turn the weak Erd\H{o}s-Szemer\'{e}di sum-product conjecture. We prove Theorem~\ref{thm:BS_almost_subgroups} in Section~\ref{sec:conclusion}.

%\fxnote{We need the $Lambda$ machinery in order to prove it}

\begin{Remark}
   It is highly probable that our method can be combined with the ideas of \cite{bourgain2009sum} which would generalize Theorem~\ref{thm:BS_almost_subgroups} to $K$-almost subgroups consisting of algebraic numbers of degree at most $d$ (though not necessarily contained in the same field extension). The upper power $C$ is going to depend on $d$ then, so the putative bound (using the notation of Theorem~\ref{thm:BS_almost_subgroups}) is 
   $$
      A(2, K) \leq C'(d)|A|^\gamma K^{C(\gamma, d)}
   $$
with some $C, C' > 0$.
We are going to consider this matter in detail elsewhere. Note, however, that proving a similar statement with no dependence on $d$ seems to be a significantly harder problem. 
\end{Remark}

\subsection{Further applications}

\subsubsection{An inverse Szemer\'{e}di-Trotter Theorem} Theorem \ref{thm:BS_almost_subgroups} can be interpreted as a partial inverse to the Szemer\'{e}di-Trotter Theorem. The Szemer\'{e}di-Trotter Theorem states that, if $P$ is a finite set of points and $L$ is a finite set of lines in $\mathbb R^2$, then the number of incidences $I(P,L)$ between $P$ and $L$ satisfies the bound
\begin{equation}
I(P,L):= |\{(p,l) \in P \times L : p \in l \}| =O(|P|^{2/3}|L|^{2/3} +|P| + |L| ).
\label{eq:ST}
\end{equation}
The term $|P|^{2/3}|L|^{2/3}$ above is dominant unless the sizes of $P$ and $L$ are rather imbalanced. The Szemer\'{e}di-Trotter Theorem is tight, up to the multiplicative constant.  

It is natural to consider the inverse question: for what sets $P$ and $L$ is it possible that $I(P,L) = \Omega (|P|^{2/3}|L|^{2/3})$? The known constructions of point sets which attain many incidences appear to all have some kind of lattice like structure. This perhaps suggests the loose conjecture that point sets attaining many incidences must always have some kind of additive structure, although such a conjecture seems to be far out of reach to the known methods.

However, with an additional restriction that $P=A \times A$ with $A \subset \mathbb Q$, Theorem \ref{thm:mainmain} leads to the following partial inverse theorem, which states that if $A$ has small product set then $I(P,L)$ cannot be maximal.

\begin{Theorem} \label{thm:STinverse} For all $\gamma \geq 0$ there exists a constant $C=C(\gamma)$ such that the following holds. Let $A$ be a finite set of rationals such that $|AA| \leq K|A|$ and let $P= A \times A$. Then, for any finite set $L$ of lines in the plane, $I(P,L)\leq 3 |P| + |A|^{\gamma}K^C|L|$.
\end{Theorem}

In fact, not only does this show that $I(A \times A,L)$ cannot be maximal when $|AA|$ is small, but better still the number of incidences is almost bounded by the trivial linear terms in \eqref{eq:ST}. The insistence that the point set is a direct product is rather restrictive. However, since many applications of the Szemer\'{e}di-Trotter Theorem make use of direct products, it seems likely that Theorem \ref{thm:STinverse} could be useful. The proof is given in Section \ref{sec:applications}.

\subsubsection{Improved bound for the size of an additive basis of a set with small product set} 
%\fxnote{Olly: Nice result! I tinkered with the notation because I didn't like that C was a set and also a constant. Made a few other small changes to the presentation. Have moved the proof to a new section at the end of the paper called "Further applications". By the way, does this imply something better for $E_*(A+A)$?}

Theorem~\ref{thm:BS_almost_subgroups} also yields the following application concerning the problem of bounding the size of an additive basis considered in \cite{shkredov2016additive}. We can significantly improve the bound in the rational setting, pushing the exponent in (\ref{eq:basis_bound}) from $1/2 + 1/442 - o_\epsilon(1)$ to $2/3 - o_\epsilon(1)$ in the limiting case $K = |A|^{\epsilon}$.

\begin{Theorem} \label{thm:additivebasis}
For any $\gamma > 0$ there exists $C(\gamma)$ such that
for an arbitrary $A \subset \mathbb{Q}$ with $|AA| = K|A|$ and $B, B' \subset \mathbb{Q}$,
$$
S := \left|\{(b, b') \in B \times B' : b + b' \in A\} \right| \leq 2|A|^\gamma K^C \min \{|B|^{1/2}|B'| + |B|, |B'|^{1/2}|B| + |B'| \}. 
$$
In particular, for any $\gamma > 0$ there exists $C(\gamma)$ such that if $A \subset B + B$ then

\begin{equation} \label{eq:basis_bound}
|B| \geq |A|^{2/3 - \gamma}K^{-C}.
\end{equation}

\end{Theorem}

The proof of Theorem \ref{thm:additivebasis} is given in Section \ref{sec:applications}.
\begin{Remark}
 During the preparation of the manuscript we became aware that Cosmin Pohoata has independently proved Theorem~\ref{thm:additivebasis} using an earlier result of Chang and by a somewhat different method.
\end{Remark}

\subsubsection{Unlimited growth for products of difference sets} It was conjectured in \cite{BRNZ} that for any $b \in \mathbb R$ there exists $k=k(b) \in \mathbb N$ such that for all $A \subset \mathbb R$
\[|(A-A)^k| \geq |A|^b .\]
In another application of Theorem \ref{thm:mainmain}, we give a positive answer to this question under the additional restriction that $A \subset \mathbb Q$. In fact, we prove the following stronger statement.

%\fxnote{I believe that we can prove an even stronger version: 
%\[|(A+B)^k| \geq |A|^b .\]
%whenever $|B| > 1$.
%}

\begin{Theorem} \label{thm:proddiff}
For any $b \in \mathbb R$ there exists $k=k(b) \in \mathbb N$ such that for all $A \subset \mathbb Q$ and $B \subset \mathbb Q$ with $|B| \geq 2$,
\[|(A+B)^k| \geq |A|^b .\]

\end{Theorem}
The proof is given in Section \ref{sec:applications}.

\subsection{Asymptotic notation}

Throughout the paper, the standard notation
$\ll,\gg$ is applied to positive quantities in the usual way. Saying $X\gg Y$ or $Y \ll X$ means that $X\geq cY$, for some absolute constant $c>0$.  The expression $X \approx Y$ means that both $X \gg Y$ and $X \ll Y$ hold.

\subsection{The structure of the rest of this paper}

In section \ref{sec:energy}, we introduce a new kind of mixed energy, and establish some initial bounds on this energy which are strong when the set $A$ is defined by relatively few primes ($c\log |A|$ for a sufficiently small constant $c$). The structure of these arguments are similar to those introduced by Chang in \cite{C}, and also used by the authors in \cite{HRNZ}. 

The goal of section \ref{sec:lambda} is to prove the main technical result of the paper, Theorem \ref{thm:lambda}.  The statement uses the language of $\Lambda$-constants, which is a robust generalisation of additive energy, and so we must first define what these constants are and identify some of their crucial properties. We also introduce the notion of query complexity, which is nicely tuned in to the techniques used and results established in Section \ref{sec:energy}. An essential tool in converting the bounds from Section \ref{sec:energy} into strong bounds for $\Lambda$-constants is a deep new result of P\"{a}vl\"{o}lgyi and Zhelezov \cite{PZ}.

In section \ref{sec:conclusion}, we use Theorem \ref{thm:lambda} to conclude the proofs of the main results of this paper, Theorems \ref{thm:mainmain}, \ref{thm:main1} and \ref{thm:BS_almost_subgroups}. Finally, in Section \ref{sec:applications}, we give proofs of further applications of our main results.

\section{A Chang-type bound for the mixed energy} \label{sec:energy}

Different kinds of energies play a pivotal role in the work of Chang \cite{C} and Bourgain-Chang \cite{BC}, as well as \cite{HRNZ}. In \cite{C}, it was proved that, for any finite set of rationals $A$ with $|AA| \leq K|A|$, the \textit{k-fold additive energy}, which is defined as the number of solutions to
\begin{equation} \label{changenergy}
a_1 + \cdots + a_k = a_{k+1} + \cdots a_{2k}, \,\,\,\,\,\,\,\, (a_1,\dots,a_{2k}) \in A^{2k}, 
\end{equation}
is at most $(2k^2-k)^{kK}|A|^{k}$. A simple application of the Cauchy-Schwarz inequality then implies that the \textit{$k$-fold sum set} satisfies the bound
\[ |kA| \geq \frac{|A|^k}{(2k^2-k)^{kK}} .\]
Bound \eqref{changenergy} is close to optimal when $K=c \log |A|$, but becomes trivial when $K=|A|^{\eps}$. In \cite{BC}, (a weighted version of) this bound was used as a foundation, and developed considerably courtesy of some intricate decoupling arguments, in order to prove a bound for the $k$-fold additive energy which remains very strong when $K$ is of the order $|A|^{\eps}$.

In \cite{HRNZ}, we followed a similarly strategy to that of \cite{C}, proving that for any finite set of rationals $A$ with $|AA| \leq K|A|$ and any non-zero rational $u$, the \textit{k-fold multiplicative energy} of $A+u$, which is defined as the number of solutions to
\begin{equation}
(a_1 + u) \cdots  (a_k+u) = (a_{k+1} + u) \cdots (a_{2k} +u ), \,\,\,\,\,\,\,\, (a_1,\dots,a_{2k}) \in A^{2k}, 
\label{changprequel}
\end{equation}
is at most $(Ck^2)^{kK}|A|^{k}$. Unfortunately, in adapting the approach of \cite{C} in order to bound the number of solutions to \eqref{changprequel} in \cite{HRNZ}, we encountered some difficulties with dilation invariance which made the argument rather more complicated, and we were unable to marry our methods with those of \cite{BC} to obtain a strong bound when $K$ is of order $|A|^{\eps}$.

In this paper, we modify the approach of \cite{HRNZ} by working with a different form of energy. Consider the following representation function:
\[r_k(x,y)=|\{(a_1,\ldots,a_k)\in A^k  :a_1\cdots a_k=x,\ (a_1+u)\cdots(a_k+u)=y\}|.\]
Then, because $r_k$ is supported on $A^{(k)}\times (A+u)^{(k)}$, it follows from the Cauchy-Schwarz inequality that
\begin{equation}
|A|^{2k}=\left(\sum_{(x,y)\in A^{(k)}\times (A+u)^{(k)} }r_k(x,y)\right)^2\leq |A^{(k)}||(A+u)^{(k)}|\sum_{(x,y)\in A^{(k)}\times (A+u)^{(k)}} r_k(x,y)^2.
\label{CSbasic}
\end{equation}
The latter sum is the quantity
\[\tilde E_k(A;u):=\left|\left\{(a_1,\ldots,a_k,b_1,\ldots,b_k)\in A^{2k} :\prod_{i=1}^ka_i=\prod_{i=1}^kb_i,\ \prod_{i=1}^k(a_i+u)=\prod_{i=1}^k(b_i+u)\right\}\right|.\]

We summarise this in the following lemma.
\begin{Lemma} \label{lem:CSbasic}
For any finite set $A\subset \mathbb R$, any  $u \in \mathbb R \setminus \{0\}$ and any integer $k \geq 2$, we have
\[|A|^{2k} \leq |A^{(k)}||(A+u)^{(k)}| \tilde E_k(A;u) .\]
In particular,
\[\frac{|A|^k}{\tilde E_k(A;u)^{1/2}}\leq \max\{|A^{(k)}|,|(A+u)^{(k)}|\}.\]
\end{Lemma}

Our goal is to estimate this energy and to show that, at least for sets of rationals, it cannot ever be too big.

In this section we seek to give an initial upper bound for $\tilde E_k(A;u)$. The strategy is close to that of Chang \cite{C}. There are also clear similarities with the prequel to this paper \cite{HRNZ}.

To do this, as in \cite{HRNZ}, we will write $\tilde E_k(A;u)$ in terms of Dirichlet polynomials. In this case, our Dirichlet polynomials will be functions of the form \[F(s_1,s_2)=\sum_{(a,b)\in\QQ^2}\frac{f(a,b)}{a^{s_1}b^{s_2}}\] where $f:\QQ^2\to\CC$ is some function of finite support. It will also be more convenient to count weighted energy. For $w_a$ a sequence of non-negative weights on $A$, let
\[\tilde E_{k,w}(A;u)=\sum_{\substack{a_1\cdots a_k=b_1\cdots b_k\\ (a_1+u)\cdots(a_k+u)=(b_1+u)\cdots(b_k+u)}}w_{a_1}\cdots w_{a_k}w_{b_1}\cdots w_{b_k}\]

\begin{Lemma} \label{lem:direnergy}
Let $A$ be a finite set of rational numbers and let $u$ be a non-zero rational number. Then, for any integer $k \geq 2$, we have
\[\tilde E_{k,w}(A;u)=\lim_{T\to \infty}\frac{1}{T^2}
\int_0^T\int_0^T\left|\sum_{a\in A} w_aa^{it_1}(a+u)^{it_2}\right|^{2k}dt_1dt_2.\]
\end{Lemma}
\begin{proof}
Expanding, the double integral on the right hand side is equal to
\begin{multline*}
\sum_{a_1,\ldots,a_k\in A}\sum_{b_1,\ldots,b_k\in A}w_{a_1}\cdots w_{a_k}w_{b_1}\cdots w_{b_k}\cdot\\
\cdot\int_0^T (a_1\cdots a_kb_1^{-1}\cdots b_k^{-1})^{it_1}dt_1\int_0^T((a_1+u)\cdots (a_k+u)(b_1+u)^{-1}\cdots (b_k+u)^{-1})^{it_2}dt_2.\end{multline*}
Now
\[\frac{1}{T}\int_0^T (u/v)^{it}dt=\begin{cases}1&\text{ if }u=v,
\\ O_{u,v}(T^{-1})&\text{ if }u\neq v.\end{cases}\]
From this, the lemma follows.
\end{proof}

%\fxnote{Brandon, please can you give a precise definition of a Dirichlet polynomial somewhere around here? It is part of the statement of Lemma \ref{lem:split}, but actually I am not sure of the definition. Or alternatively, perhaps we could modify the statement of Lemma \ref{lem:split} to hold for any suitably integrable function? Done and done.}
%\fxnote{Olly: I made a correction to the definition of the norm, someone please check that I didn't mess it up!}
Let $\|\cdot \|_{2k}$ be the standard norm in $L^{2k}([0, T]^2)$, normalised such that $\| 1\|_{2k} = 1$. So,
\[
\| f \|_{2k} := \left( \frac{1}{T^2} \int_0^T \int_0^T|f(t)|^{2k} dt \right)^{1/2k}. 
\]

\begin{Lemma} \label{lem:split}
Let $\cJ$ be a set of integers and decompose it as $\cJ=\cJ_1 \cup \cdots \cup\cJ_N$. For each $j \in \cJ$ let $f_j:\RR\times \RR\to \CC$ be a function belonging to $L^{2k}\lr{\RR^2}$ for every integer $k\geq 2$. Then, for every integer $k \geq 2$,
\begin{multline}\label{split}\lim_{T\to \infty}\lr{\frac{1}{T^2}\int_0^T\int_0^T\left|\sum_{j\in\cJ} f_j(t_1,t_2)\right|^{2k}dt_1dt_2}^{1/k} \\\leq N\sum_{n=1}^N \lim_{T\to \infty}\lr{\frac{1}{T^2}\int_0^T\int_0^T\left| \sum_{j\in \cJ_n}f_j(t_1,t_2)\right|^{2k}dt_1dt_2}^{1/k}.
\end{multline}
\end{Lemma}

\begin{proof} 
It suffices to prove the inequality for all sufficiently large $T$, which we assume fixed for now.
Then
\begin{equation} \label{eq:Lknormsum}
\lr{\frac{1}{T^2}\int_0^T\int_0^T\left|\sum_{j\in\cJ} f_j(t_1,t_2)\right|^{2k}dt_1dt_2}^{1/k} =
 \left(\left\| \sum_{n=1}^N\sum_{j \in \cJ_n}  f_j \right\|_{2k} \right)^2\leq\left(\sum_{n=1}^N\left\| \sum_{j \in \cJ_n}  f_j \right\|_{2k} \right)^2,
\end{equation}
by the triangle inequality. By the Cauchy-Schwarz inequality, (\ref{eq:Lknormsum}) is bounded by 
\begin{equation}
N\sum_{n=1}^N\left\| \sum_{j \in \cJ_n}  f_j \right\|_{2k}^2.
\end{equation}
Letting $T \to \infty$ we get the claim of the lemma.
\end{proof}
\begin{Corollary}\label{Split}
Let $A$ be a finite set of rational numbers, partitioned as $A=A_1\cup\cdots\cup A_N$, let $w$ be a set of non-negative weights, and let $u$ be a non-zero rational number. Then for any integer $k \geq 2$
\[\tilde E_{k,w}(A;u)^{1/k}\leq N\sum_{j=1}^N\tilde E_{k,w}(A_j;u)^{1/k}.\]
\end{Corollary}
Now let $p$ be a fixed prime. For $a \in \mathbb Q$, let $v_p(a)$ denote the $p$-adic valuation of $a$. For a set $A$ of rational numbers and an integer $t$, we let 
$A_t=\{a\in A:v_p(a)=t\}$.

\begin{Lemma} \label{thm:basecase}
Let $p$ be a prime number. Suppose $A$ is a finite set of rational numbers and let $u$ be a non-zero rational number.
Then for any $w$, a set of non-negative weights on $A$, and any integer $k \geq 2$,
\[\tilde E_{k,w}(A;u)^{1/k}\leq 2\binom{2k}{2}\sum_{d\in \ZZ}\tilde E_{k,w}(A_d;u)^{1/k}.\]
\end{Lemma}

%\fxnote{Olly: I have expanded the proof of this lemma quite a lot, based on some copy and paste from our first paper.}

\begin{proof}
First, let $A=A_+\cup A_-$ where $A_+=\{a\in A:v_p(a)\geq v_p(u)\}$ and $A_-=\{a\in A:v_p(a)<v_p(u)\}$. By Corollary \ref{Split}, we have
\begin{equation}
\tilde E_{k,w}(A;u)^{1/k}\leq 2\tilde E_{k,w}(A_+;u)^{1/k}+2\tilde E_{k,w}(A_-;u)^{1/k}.
\label{+-split}
\end{equation}
These two terms will be dealt with in turn, starting with $E_{k,w}(A_+;u)^{1/k}$. 
\begin{comment}Suppose we have a solution
\[(a_1+u)\cdots(a_k+u)=(b_1+u)\cdots(b_k+u)\]
and so
\begin{equation}
(a_1u^{-1}+1)\cdots(a_ku^{-1}+1)=(b_1u^{-1}+1)\cdots(b_ku^{-1}+1).
\label{eq:normal}
\end{equation}
We claim that $v_p$ cannot have a unique minimal value among the $a_i$ and $b_i$. Indeed, since $v_p(a_iu^{-1})\geq 0$ and $v_p(b_iu^{-1})\geq 0$, expanding out both sides of \eqref{eq:normal} and simplifying
\[u^{-1}(a_1+\cdots+a_k)+\text{higher terms}=u^{-1}(b_1+\cdots+b_k)+\text{higher terms},\] 
and if $v_p$ achieves a unique minimum, then the left hand side and the right hand side are divisible by distinct powers of $p$, a contradiction.
\end{comment}
To do this, we first set up some more notation. For an integer $d$, define the function
\[f_d(t_1,t_2):= \sum_{a \in A_d} w_a a^{it_1} (a+u)^{it_2} .\]
Then, by Lemma \ref{lem:direnergy}
\[\tilde E_{k,w}(A_+;u) = \lim_{T \rightarrow \infty} \frac{1}{T^2} \int_0^T \int_0^T \left|\sum_{d \geq v_p(u)} f_d(t_1,t_2)\right|^{2k} dt_1 dt_2 .\]
Expanding this expression, as in the proof of Lemma \ref{lem:direnergy}, we obtain that $\tilde E_{k,w}(A_+;u)$ is equal to
\begin{equation}
\sum_{d_1,\dots,d_{2k} \geq v_p(u)}\lim_{T \rightarrow \infty} \frac{1}{T^2} \int_0^T \int_0^T  f_{d_1}(t_1,t_2)\cdots f_{d_k}(t_1,t_2) \bar{f_{d_{k+1}}(t_1,t_2)}\cdots \bar{f_{d_{2k}}(t_1,t_2)} dt_1 dt_2 .
\label{eq:long}
\end{equation}

For fixed $d_1, \dots, d_{2k}$, the quantity
\[\lim_{T \rightarrow \infty} \frac{1}{T^2} \int_0^T \int_0^T  f_{d_1}(t_1,t_2)\cdots f_{d_k}(t_1,t_2) \bar{f_{d_{k+1}}(t_1,t_2)}\cdots \bar{f_{d_{2k}}(t_1,t_2)} dt_1 dt_2 .\]
gives a weighted count of the number of solutions to the system of simultaneous equations
\begin{align}
a_1 \cdots a_k &= a_{k+1} \cdots a_{2k} \label{basiceq1}
\\ (a_1+u) \cdots (a_k+u) &= (a_{k+1}+u) \cdots (a_{2k}+u),
\label{basiceq2}
\end{align}
such that $a_i \in A_{d_i}$.

We claim that there are no solutions to \eqref{basiceq2}, and thus also no solutions to the above system, if all of the $d_i$ are distinct. Indeed, suppose we have a solution
\[(a_1+u)\cdots(a_k+u)=(a_{k+1}+u)\cdots(a_{2k}+u)\]
and so
\begin{equation}
(a_1u^{-1}+1)\cdots(a_ku^{-1}+1)=(b_{k+1}u^{-1}+1)\cdots(b_{2k}u^{-1}+1).
\label{eq:normal}
\end{equation}

Since $v_p(a_iu^{-1})\geq 0$, expanding out both sides of \eqref{eq:normal} and simplifying gives
\begin{equation}
u^{-1}(a_1+\cdots+a_k)+\text{higher terms}=u^{-1}(b_{k+1}+\cdots+b_{2k})+\text{higher terms}.
\label{eq:contradiction}
\end{equation} 
If all of the $d_i$ are distinct, then there is some unique smallest $d_i$, and thus a unique smallest value of $v_p(a_i)$. But then the left hand side and the right hand side are divisible by distinct powers of $p$, a contradiction.

So returning to \eqref{eq:long}, we need only consider the cases in which one or more of the $d_i$ are repeated. There are three kinds of ways in which this can happen.

\begin{enumerate}
\item $d_i=d_i'$ with $1 \leq i \leq k$ and $k+1 \leq i' \leq 2k$. There are $k^2$ possible positions for such a pair $(i,i')$,
\item $d_i=d_i'$ with $1 \leq i, i' \leq k$. There are $\binom{k}{2}$ possible positions for such a pair $(i,i')$,
\item $d_i=d_i'$ with $k+1 \leq i, i' \leq 2k$. There are $\binom{k}{2}$ possible positions for such a pair $(i,i')$.
\end{enumerate}

Suppose we are in situation (1) above. Specifically, suppose that $d_1=d_{2k}$. The other $k^2-1$ cases can be dealt with by the same argument. Then these terms in \eqref{eq:long} can be rewritten as

\begin{multline}
\sum_{d_1 \geq v_p(u)}\lim_{T \rightarrow \infty} \frac{1}{T^2} \int_0^T \int_0^T  f_{d_1}(t_1,t_2)  \bar{f_{d_1}(t_1,t_2)} 
\\ \sum_{d_2,\dots,d_{2k-1} \geq v_p(u)} f_{d_2}(t_1,t_2) \cdots f_{d_k}(t_1,t_2) \bar{f_{d_{k+1}}(t_1,t_2)}\cdots \bar{f_{d_{2k-1}}(t_1,t_2)} dt_1 dt_2 
\\= \sum_{d \geq v_p(u)}\lim_{T \rightarrow \infty} \frac{1}{T^2} \int_0^T \int_0^T  |f_{d}(t_1,t_2)   |^2 \left |\sum_{d \geq v_p(u)} f_d(t_1,t_2)\right |^{2(k-1)} dt_1 dt_2.
\end{multline}

Suppose we are in situation (2). Specifically, suppose that $d_1=d_{2}$. The other $ \binom{k}{2}-1$ cases can be dealt with by the same argument. Then these terms in \eqref{eq:long} can be rewritten as

\begin{align*} & \sum_{d_1 \geq v_p(u)}  \lim_{T\to \infty} \frac{1}{T^2} \int_0^T \int_0^T  f_{d_1}^2(t_1,t_2)  \sum_{d_3,\dots, d_{2k} \geq v_p(u)} f_{d_3}(t_1,t_2) \cdots f_{d_k}(t_1,t_2)  \bar{f_{d_{k+1}}(t_1,t_2)} \cdots \bar{f_{d_{2k}}(t_1,t_2)} dt_1 dt_2 
\\& \leq  \sum_{d \geq v_p(u)} \lim_{T\to \infty}\frac{1}{T^2}\int_0^T\int_0^T   |f_d(t_1,t_2)|^2\left|\sum_{d \geq v_p(u)} f_d(t_1,t_2)\right|^{k-2} \left|\sum_{d} \bar{f_d(t_1,t_2)}\right|^{k}  dt_1dt_2
\\&= \sum_{d \geq v_p(u)} \lim_{T\to \infty}\frac{1}{T^2}\int_0^T \int_0^T   |f_d(t_1,t_2)|^2\left|\sum_{d \geq v_p(u)} f_d(t_1,t_2)\right|^{2(k-1)}   dt_1dt_2.
\end{align*}

The same argument also works in case (3). Returning to \eqref{eq:long}, we then have
\begin{align*}
\tilde E_{k,w}(A_+;u)& \leq 
 \binom{2k}{2}\sum_{d \geq v_p(u)} \lim_{T\to \infty}\frac{1}{T^2}\int_0^T \int_0^T   |f_d(t_1,t_2)|^2\left|\sum_{d \geq v_p(u)} f_d(t_1,t_2)\right|^{2(k-1)}   dt_1dt_2\\
&\leq \binom{2k}{2}\sum_{d\geq v_p(u)}\tilde E_{k,w}(A_d;u)^{1/k}E_{k,w}(A_+;u)^{1-1/k},
\end{align*}
%\fxnote{we need to expand the explanation of the first inequality above.}
the last inequality being H\"older's. It therefore follows that
\begin{equation}
\tilde E_{k,w}(A_+;u)^{1/k} \leq \binom{2k}{2} \sum_{d\geq v_p(u)}\tilde E_{k,w}(A_d;u)^{1/k}  .
\label{plusbound}
\end{equation}
Now we proceed to $E_{k,w}(A_-;u)^{1/k}$. For any solution to the pair of equations
\begin{align*}
a_1\cdots a_k&=a_{k+1}\cdots a_{2k}\\ (a_1+u)\cdots (a_k+u)&=(a_{k+1}+u)\cdots (a_{2k}+u)
\end{align*}
we have a solution to the equation
\[(1+ua_1^{-1})\cdots (1+ua_k^{-1})=(1+ua_{k+1}^{-1})\cdots (1+ua_{2k}^{-1}).\]
Again, we expand and simplify, using this time that $v_p(ua_i^{-1})$ is positive, and get
\[u(a_1^{-1}+\cdots a_k^{-1})+\text{higher terms}=u(a_{k+1}^{-1}+\cdots a_{2k}^{-1})+\text{higher terms}.\]
As in the previous case
\footnote{Note that here we have used the information that $a_1 \cdots a_k=a_{k+1} \cdots a_{2k}$, whereas we did not use this when bounding $\tilde E_{k,w}(A_+;u)$.}
, we cannot have a unique smallest $v_p(ua_i^{-1})$. We can therefore repeat the arguments that gave us \eqref{plusbound} in order to deduce that
\begin{equation}
\tilde E_{k,w}(A_-;u)^{1/k} \leq \binom{2k}{2} \sum_{d < v_p(u)}\tilde E_{k,w}(A_d;u)^{1/k}  .
\label{minusbound}
\end{equation}
Inserting \eqref{plusbound} and \eqref{minusbound} into \eqref{+-split} completes the proof.
\end{proof}

Next, this is used as a base case to give an analogous result with more primes.

\begin{Lemma} \label{thm:chang1}
Let $p_1,\dots,p_K$ be a prime numbers. Suppose $A$ is a finite set of rational numbers and let $u$ be a non-zero rational number. For a vector $\dd=(d_1,\dots,d_K)$, define
\[A_{\dd}=\{a\in A:v_{p_1}(a)=d_1, \dots, v_{p_K}(a)=d_K\}.\]
Then for any $w$, a set of non-negative weights on $A$, and for any integer $k \geq 2$,
\[\tilde E_{k,w}(A;u)^{1/k}\leq \left(2\binom{2k}{2} \right)^K\sum_{\dd\in \ZZ^K}\tilde E_{k,w}(A_{\dd};u)^{1/k}.\]
\end{Lemma}

\begin{proof}
%To follow. This will be the standard induction argument we use in the other paper. Copy and paste.

%We establish the convention that $f_{\jj}$ is identically zero for all $\jj \notin \cJ$. Therefore, we can complete the sum in \eqref{eq:mainint}, and 
The aim is to prove that
\begin{multline}
\lim_{T\to \infty}\left (\frac{1}{T^2}\int_0^T \int_0^T\left|\sum_{\dd\in \mathbb Z^K} \sum_{a \in A_{\dd}}w_a a^{it_1}(a+u)^{it_2} \right|^{2k}dt_1 dt_2 \right)^{1/k}
\\ \leq \left(2\binom{2k}{2} \right)^K\sum_{\dd \in \mathbb Z^K}\lim_{T\to \infty}\lr{\frac{1}{T^2}\int_0^T \int_0^T\left| \sum_{a \in A_{\dd}} w_a a^{it_1}(a+u)^{it_2}    \right|^{2k}dt_1 dt_2}^{1/k}.
\end{multline}

We proceed by induction on $K$, the base case $K=1$ being given by Lemma \ref{thm:basecase}. Then
\begin{align*}
&\lim_{T\to \infty} \left(\frac{1}{T^2} \int_0^T \int_0^T \left|\sum_{\dd \in \mathbb Z^K }\sum_{a \in A_{\dd}}w_a a^{it_1}(a+u)^{it_2} \right|^{2k} dt_1dt_2\right)^{1/k}
\\&=\lim_{T\to \infty} \left(\frac{1}{T^2} \int_0^T \int_0^T  \left|\sum_{d_K\in \mathbb Z} \left(\sum_{\dd'\in \mathbb Z^{K-1} }\sum_{a \in A_{(\dd',d)}}w_a a^{it_1}(a+u)^{it_2} \right) \right|^{2k} dt_1 dt_2\right)^{1/k}
\\&\leq 2\binom{2k}{2}  \sum_{d_K\in \mathbb Z} \lim_{T\to \infty} \left(\frac{1}{T^2} \int_0^T \int_0^T \left|\sum_{\dd'\in \mathbb Z^{K-1} }\sum_{a \in A_{(\dd',d)}}w_a a^{it_1}(a+u)^{it_2}  \right|^{2k} dt_1 dt_2\right)^{1/k}
\\& \leq 2\binom{2k}{2}  \sum_{d_K\in \mathbb Z} \left( 2\binom{2k}{2}  \right)^{K-1} \sum_{\dd'\in \mathbb Z^{K-1} } \lim_{T\to \infty} \left(\frac{1}{T^2} \int_0^T \int_0^T \left| \sum_{a \in A_{(\dd',d)}}w_a a^{it_1}(a+u)^{it_2}  \right|^{2k} dt_1 dt_2\right)^{1/k}
\\&=\left( 2\binom{2k}{2}  \right)^{K} \sum_{\dd \in \mathbb Z^K}\lim_{T\to \infty}\lr{\frac{1}{T^2}\int_0^T \int_0^T\left| \sum_{a \in A_{\dd}}w_a a^{it_1}(a+u)^{it_2}  \right|^{2k}dt_1 dt_2}^{1/k}.
\end{align*}
The first inequality above follows from an application of Lemma \ref{thm:basecase}. The second inequality follows from the induction hypothesis.
\end{proof}

\section{Lambda-constants and query complexity} \label{sec:lambda}

\subsection{Lambda constants}

In order to extract as much as possible from the Theorem \ref{thm:chang1}, it will be convenient to use the language of \textit{$\Lambda$-constants}. The main motivation behind  $\Lambda$-constants is the stability property given by the forthcoming Corollary~\ref{corr:stability}, which is absent in the non-weighted version of the energy. 

We also encourage the interested reader to consult our preceding paper \cite{HRNZ} for a slightly more gentle introduction to $\Lambda$-constants in the setting of Dirichlet polynomials and more in-depth motivation behind this concept. 

%\fxnote{Dmitry always writes nice introductions to these things! Perhaps you could expand here on why they are useful and interesting things?
%Dmitry: Our paper is probably going to grow beyond 50 pages, so I'd rather be lazy and cite our previous paper here :) Olly: Smart idea. }

Let $A \subset \mathbb Q$ be a finite set and let $u$ be a non-zero rational. Define
\[\Lambda_k(A;u):= \max \tilde E_{k,w}(A;u)^{1/k},\]
where the maximum is taken over all weights $w$ on $A$ such that
\begin{equation}
\sum_{a \in A} w(a)^2 =1. 
\label{eq:weightsnorm}
\end{equation}
An equivalent definition is
\[\Lambda_k(A;u):= \max \lim_{T\to \infty}
\left \|\sum_{a\in A}w_aa^{it_1}(a+u)^{it_2} \right \|_{2k}^2.\]
where the maximum is taken over the same range of weights.

\begin{Lemma} \label{lm:anyweights}
 Let $A \subset \mathbb Q$ be a finite set with some non-negative real weights $w_a$ assigned to each element $a \in A$ and let $u$ be a non-zero rational. Then
 \begin{equation} \label{eq:lambdaanyweights}
 \left \| \sum_{a \in A} w_a a^{it_1}(a+u)^{it_2} \right\|^2_{2k} \leq \Lambda_k(A;u) \left(\sum_{a \in A} w_a^2 \right) + o_{T\to\infty}(1).
 \end{equation}
\end{Lemma}
\begin{proof}
If $\sum_{a \in A} w^2_a = 0$ the claim of the lemma is trivial. Otherwise, define new weights
\[
w'_a := \frac{w_a}{(\sum_{a \in A} w_a^2)^{1/2}}
\]
which satisfy (\ref{eq:weightsnorm}). It thus suffices to show that
\[
\left \| \sum_{a \in A} w'_a a^{it_1}(a+u)^{it_2} \right\|^2_{2k} \leq \Lambda_k(A;u) + o_{T\to\infty}(1),
\]
which is a straightforward consequence of our definition of $\Lambda_k(A;u)$.
\end{proof}

We will use the following stability property of $\Lambda$-constants which helps us to work with subsets.

\begin{Corollary} \label{corr:stability}
Suppose that $A \subset \mathbb Q$, that $u$ is a non-zero rational and $A' \subset A$. Then
\[
\Lambda_k(A'; u) \leq \Lambda_k(A; u).
\]
In particular,
\[
\tilde E^{1/k}_k(A';u) \leq \Lambda_k(A;u)|A'|.
\]
and
\[
\tilde E_k(A;u) \leq \Lambda^k_k(A;u)|A|^k.
\]
\end{Corollary}
\begin{proof}
 The first claim follows from the observation that any set of weights $\{w_a \}_{a \in A'}$ with $\sum w^2_a = 1$ can be trivially extended to a set of weights $\{ w_a \}_{a \in A}$ by assigning zero weight to the elements in $A \setminus A'$. Next observe that $E_k$ is just $E_{k, w}$ with all the weights being one and apply Lemma~\ref{lm:anyweights}.
\end{proof}

\subsection{Query complexity}

The ideas of Section \ref{sec:energy} dovetail perfectly with the notion of the \textit{query-complexity} of a set of rationals. Given a set $A \subset \mathbb Q$, we define its query complexity $q(A)$ to be the smallest integer $t$ such that there are functions $f_i: \mathbb{Z} \to \mathbb{P}, i = 1, \ldots, t-1$ and a fixed prime $p_0$ such that the vectors
\[
(v_{p_0}(a), v_{p_1}(a),\dots,v_{p_{t-1}}(a)),\,\,\,\,\, a \in A
\]
are pairwise distinct, with the primes $p_i$ defined recursively as 
\begin{equation} \label{eq:evaluation}
p_{i} = f_i(v_{p_{i-1}}(a)).
\end{equation}

In the language of computational complexity, suppose that Alice and Bob agree on a set $A \subset \mathbb Q$, and then Alice secretly chooses an element $a \in A$. Bob can recover the value $a \in A$ by querying Alice iteratively at most $t$ times, at step $i$ evaluating $p_i$ using (\ref{eq:evaluation}) and asking Alice for $v_{p_i}(a)$.

The following result was recently proven by P\"{a}vl\"{o}lgyi and Zhelezov \cite{PZ}, building on work of Matolsci, Ruzsa, Shakan and Zhelezov \cite{MRSZ}.\footnote{We state a version of the result which is geared towards the particular considerations of our problem; see \cite[Theorem 1.1]{PZ} for a more general statement.}

\begin{Theorem} \label{thm:PFR} For any $\epsilon>0$, and any set $A \subset \mathbb Q$ with $|AA|\leq K|A|$, there exists a subset $A' \subset A$ with $|A'| \geq K^{-\frac{2}{\epsilon}}|A|$ and $q(A) \leq \epsilon \log_2|A|$.

\end{Theorem}

The next lemma records that any set with small query complexity also has a small $\Lambda$-constant.
\begin{Lemma} \label{lem:lambdasep}
Let $A \subset \QQ$ with $q(A) \leq t$. Then for any $u \in \QQ \setminus \{0\}$
\[ \Lambda_k(A;u) \leq \left(2\binom{2k}{2} \right)^{t} .\]
\end{Lemma}

\begin{proof} Write $t=q(A)$. Let $w$ be any set of weights on $A$ that satisfy \eqref{eq:weightsnorm}. Let $a \in A$ be arbitrary. In the notation of Lemma \ref{thm:chang1}, we have a list of primes $p_1, p_2, \dots,p_t$ defined by (\ref{eq:evaluation}) such that the set
\[
A_{\dd}=\{a' \in A :v_{p_1}(a')=v_{p_1}(a),\dots,v_{p_t}(a')=v_{p_t}(a)\}
\]
has cardinality exactly $1$. For any singleton $\{a\} \in A$, $\tilde E_{k,w}(\{a\};u)=w_a^{2k}$. Therefore, by Lemma \ref{thm:chang1},
\[\tilde E_{k,w}(A';u)^{1/k}\leq \left(2\binom{2k}{2} \right)^{t}\sum_{a \in A'}w_a^2=\left(2\binom{2k}{2} \right)^{t}.\]
\end{proof}

The following result is important generalisation of the previous one; it shows that if $A$ contains a large subset with small query complexity then $A$ itself has small $\Lambda$-constant.

\begin{Lemma} \label{lem:subset}
Let $A \subset \mathbb Q^*$ be a finite set with $|AA| \leq K|A|$ and let $u$ be a non-zero rational number. Suppose that $A' \subset A$ and $q(A')=t$. Then
\[ \Lambda_k(A;u) \leq K^4 \left(\frac{|A|}{|A'|}\right)^2 \left(2\binom{2k}{2} \right)^{t}. \]
\end{Lemma}

\begin{proof} %To follow (see Proposition 16 in Brendan-Dmitry exposition).

Let $w$ be an arbitrary set of weights on $A$ such that $\sum_{a \in A} w(a)^2 =1$. We seek a suitable upper bound for
\[ \left\|  \sum_{a\in A}w_aa^{it_1}(a+u)^{it_2} \right \|_{2k}^2 .\]
For a fixed $z \in A/A'$, define a set of weights $w^{(z)}$ on $zA'$ by taking $w^{(z)}(za')=w(za')$ if $za' \in A$ and $w^{(z)}(za')=0$ otherwise. Define 
\[R_{(A/A'),A'}(x):=|\{(s,a) \in (A/A') \times A' : sa=x \}| \]
and note that $R_{(A/A'),A'}(x) \geq |A'|$ for all $x \in A$. This is because, for all $a' \in A'$, $x=(\frac{x}{a'})a'$. Therefore,
\begin{align*} \left\| \sum_{z \in A/A'} \sum_{a' \in A'} w^{(z)}(za')(za')^{it_1}(za'+u)^{it_2} \right \|_{2k}
&= \left\| \sum_{a \in A}R_{(A/A'),A'} (a) w(a)a^{it_1}(a+u)^{it_2} \right \|_{2k}
\\& \geq |A'|\left\|  \sum_{a\in A}w_aa^{it_1}(a+u)^{it_2} \right \|_{2k}.
\end{align*}
On the other hand, by the triangle inequality and Lemma \ref{lm:anyweights}
\begin{align*}
\left\| \sum_{z \in A/A'} \sum_{a' \in A'} w^{(z)}(za')(za')^{it_1}(za'+u)^{it_2} \right \|_{2k}
& \leq \sum_{z \in A/A'} \left\|  \sum_{a' \in A'} w^{(z)}(za')(za')^{it_1}(za'+u)^{it_2} \right \|_{2k}
\\& \leq  \sum_{z \in A/A'} \Lambda_k(zA'; u)^{1/2} + o_{T \rightarrow \infty}(1).
\end{align*}
Since $q(A')=t$, it follows from Lemma \ref{lem:lambdasep} that $\Lambda_k(zA';u) = \Lambda_k(A';u/z) \leq \left(2\binom{2k}{2} \right)^{t} $. We also have
\[|A/A'| \leq |A/A| \leq \frac{|AA|^2}{|A|} \leq K^2 |A|, \]
by the Ruzsa Triangle Inequality (see \cite{TV}). It therefore follows that
\[ \left\|  \sum_{a\in A}w_aa^{it_1}(a+u)^{it_2} \right \|_{2k} \leq K^2 \left( \frac{|A|}{|A'|} \right ) \left(2\binom{2k}{2} \right)^{t/2} + o_{T \rightarrow \infty} (1),
\]
and the result follows.
\end{proof}

Combining this with Theorem \ref{thm:PFR} gives the following, which is our main result concerning $\Lambda$-constants.

\begin{Theorem} \label{thm:lambda}
Given $0<\gamma < 1/2$, there exists a positive constants $C=C(\gamma, k)$ such that for any finite $A \subset \mathbb Q^*$ with $|AA|=K|A|$ and any non-zero rational $u$,
\[ \Lambda_k(A;u) \leq K^C|A|^{\gamma} .\]
\end{Theorem}

\begin{proof}

Apply Theorem \ref{thm:PFR} with $\epsilon=\frac{\gamma}{\log_2(4k)}$. There exists $A' \subset A$ with $|A'| \geq K^{-\frac{2}{\epsilon}}|A|$ and $q(A) \leq \epsilon \log_2|A|$. Then by Lemma \ref{lem:subset}
\begin{align*}
 \Lambda_k(A;u) &\leq K^4 \left(\frac{|A|}{|A'|}\right)^2 \left(2\binom{2k}{2} \right)^{\epsilon \log_2|A|}
 \\& \leq K^{4+\frac{4}{\epsilon}}|A|^{\epsilon \log_2(4k)}.
 \end{align*}

\end{proof}

Observe that we can in fact take $C(\gamma,k)$ in Theorem \ref{thm:lambda} to be $4 +\frac{4\log_2(4k)}{\gamma}$.

\section{Concluding the proofs} \label{sec:conclusion}

In this section we conclude the proof of Theorem \ref{thm:mainmain}, which is the main theorem of this paper, and Theorem~\ref{thm:BS_almost_subgroups} announced in the introduction.

We will use the Pl\"{u}nnecke-Ruzsa Theorem. See \cite{P} for a simple inductive proof. Following convention, we state it using additive notation, although it will be used in the multiplicative setting.

\begin{Theorem} \label{thm:Plun} Let $A$ be a subset of a commutative additive group $G$ with $|A+A| \leq K|A|$. Then for any $h \in \mathbb N$,
\[ |hA| \leq K^h|A|. \]
\end{Theorem}

For the convenience of the reader, we restate Theorem \ref{thm:mainmain}.

\begin{Theorem} \label{thm:mainmainagain}
For all $b \in \mathbb Z$, there exists $k=k(b)$ such that for any finite set $A \subset \mathbb Q^*$ and any non-zero rational $u$,
\[  \max \{ |A^{(k)}|, |(A+u)^{(k)}| \} \geq |A|^b  \]
\end{Theorem}

\begin{proof} Fix $b$ and assume that
\[|A^{(k)}| < |A|^b \]
for some sufficiently large $k=2^l$. The value of $l$ (and thus also that of $k$) will be specified at the end of the proof. %The aim is to conlcude that
%\[|(A+u)| \geq |A|^b. \]
Since $|A^{(2^l)}| < |A|^b$, it follows that
\[ \frac{|A^{(2^l)}|}{|A^{(2^{l-1})}|} \frac{|A^{(2^{l-1})}|}{|A^{(2^{l-2})}|} \cdots \frac{|A^{(2)}|}{|A|} < |A|^{b-1}\]
and thus there is some integer $l_0 \leq l$ such that
\[ \frac{|A^{(2^{l_0+1})}|}{|A^{(2^{l_0})}|} < |A|^{\frac{b-1}{l}} .\]
Therefore, writing $k_0=2^{l_0}$ and $B=A^{(k_0)}$, we have
\[|BB| < |B| |A|^{\frac{b-1}{l}}. \]
Also, for any non-zero $\lambda \in \mathbb Q$, $|(\lambda B)(\lambda B)| <|B||A|^{\frac{b-1}{l}}$. Therefore, by Theorem \ref{thm:lambda},
\[ \Lambda_h(\lambda B;u) \leq |A|^{C\frac{b-1}{l}} |B|^{\gamma} \leq |A|^{C\frac{b-1}{l}+ \gamma b}  \]
where $C=C(h, \gamma)$ and $h, \gamma$ will be specified later.

Now, for some $\lambda \in \mathbb Q$, we have $A \subset \lambda B$, and thus by Corollary \ref{corr:stability} and Lemma \ref{lem:CSbasic}
\[ \frac{|A|^2}{\max \{|A^{(h)}|,|(A+u)^{(h)}| \}^{2/h}} \leq \tilde E_h^{1/h}(A;u) \leq |A| \Lambda_h(\lambda B;u) \leq |A|^{1+C\frac{b-1}{l}+ \gamma b} .\]
This rearranges to
\[ \max \{|A^{(h)}|,|(A+u)^{(h)}| \} \geq |A|^{\frac{h}{2}(1-C\frac{b-1}{l}- \gamma b)} .\]
Choose $\gamma=1/100b$ and $h=4b$. Then $C=C(h,\gamma)=C(b)$ and we have
\[ \max \{|A^{(h)}|,|(A+u)^{(h)}| \} \geq |A|^{\frac{h}{2}(99/100-C(b)\frac{b-1}{l})} .\]
Then choose $l=(b-1)4C$ to get
\[ \max \{|A^{(h)}|,|(A+u)^{(h)}| \} \geq |A|^{\frac{h}{4}}=|A|^b .\]
Note that the choice of $l$ depends only on $b$ and thus $k=2^{4C(b-1)}=k(b)$.
In particular, since $k > h$, we conclude that
\[ \max \{|A^{(k)}|,|(A+u)^{(k)}| \} \geq |A|^b ,\]
as required. 

\end{proof}

If we use the value of $C(\gamma,k)$ indicated at the end of the proof of Theorem \ref{thm:lambda} to keep track of the constants in this argument, it follows that we can take $k=2^{O(b^2\log b)}$. To be even more precise, it gives
\[
k=(16b)^{1616b^2}.
\]
This compares favourably with the dependency in the corresponding sum-product bound of Bourgain and Chang \cite{BC}, where they commented that it was possible to take $k=2^{O(b^4)}$. A similar quantitative improvement for the classical iterated sum-product problem is possible by studying the recent paper of P\"{a}vl\"{o}lgyi and Zhelezov \cite{PZ} and filling in some extra details.

Theorem \ref{thm:lambda} also implies Theorem \ref{thm:main1}. The statement is repeated below for the convenience of the reader.

\begin{Theorem}
Given $0<\gamma < 1/2$ and any integer $k \geq 2$, there exists a positive constant $C=C(\gamma, k)$ such that for any finite $A \subset \mathbb Q^*$ with $|AA|=K|A|$ and any non-zero rational $u$,
\[ | (A+u) ^{(k)}|  \geq \frac{|A|^{k(1-\gamma)-1}}{K^{Ck}}.\]
\end{Theorem}

\begin{proof} Define $w(a)= 1 / |A|^{1/2}$ for all $a \in A$ and note that \eqref{eq:weightsnorm} is satisfied. Furthermore, for this set of weights $w$,
\begin{equation}
\tilde E_{k,w}(A;u) = \frac{\tilde E_k(A;u)}{|A|^{k}} \geq \frac{|A|^{k}}{|A^{(k)}||(A+u)^{(k)}|},
\label{eq:trivialweights}
\end{equation}
where the inequality comes from Lemma \ref{lem:CSbasic}. It follows from Theorem \ref{thm:lambda} that there exists a constant $C=C(\gamma,k)$ such that for any $u \in \mathbb Q \setminus \{0\}$, $\Lambda_{k}(A;u) \leq K^C|A|^{\gamma}$. Consequently, by the definition of $\Lambda_{k}(A;u)$,
\[ \tilde E_{k,w}(A;u) \leq K^{Ck}|A|^{\gamma k}. \]
Combining this with \eqref{eq:trivialweights}, it follows that
\begin{equation}
|A^{(k)}||(A+u)^{(k)}| \geq \frac{|A|^{k(1-\gamma)}}{K^{Ck}}.
\label{eq:energybound}
\end{equation}
Finally, since $|AA| \leq K|A|$, it follows from the Pl\"{u}nnecke-Ruzsa Theorem that $|A^{(k)}| \leq K^k|A|$. Inserting this into \eqref{eq:energybound} completes the proof.

\end{proof}

We now turn to the proof of Theorem~\ref{thm:BS_almost_subgroups}. Recall its statement.

\begin{Theorem} For any $\gamma > 0$ there is $C(\gamma) > 0$ such that for any $K$-almost subgroup $A \subset \mathbb{Q}^*$ and fixed non-zero $c_1, c_2 \in \mathbb{Q}$ the number $A(2, K)$ of solutions $(x_1, x_2) \in A^2$ to 
 $$
 c_1x_1 + c_2 x_2 = 1
 $$
 is bounded by 
 $$
 A(2, K) \leq |A|^\gamma K^C.
 $$

\end{Theorem}
\begin{proof}
	Let $S \subset A$ be the set of $x_1 \in A$ such that $c_1x_1 + c_2x_2 = 1$ for some $x_2 \in A$. Since the projection $(x_1, x_2) \to x_1$ is injective, it suffices to bound the size of $S$.
    
    Since $S \subset A$, by Theorem~\ref{thm:lambda} and
    Corollary~\ref{corr:stability} for any non-zero $u$ 
$$
\tilde E_k(S; u) \leq K^{k C(\gamma', k)}|A|^{k\gamma'} |S|^k
$$
with the parameters $0 < \gamma' < 1/2, k \geq 2$ to be taken in due course.

In particular, by Lemma~\ref{lem:CSbasic}
$$
|S|^k \leq \left(K^{k C(\gamma', k)}|A|^{k\gamma'} |S|^k \right)^{1/2} \max \{ |S^k|, |(S-1/c_1)^k| \}.
$$

On the other hand, $S \subseteq A$ and $(S - 1/c_1) \subseteq -(c_2/c_1)A$, so by the Pl\"{u}nnecke-Ruzsa inequality
$$
\max \{ |S^k|, |(S-1/c_1)^k| \} \leq |A^{(k)}| \leq K^k|A|.
$$

We then have
$$
|S| \leq |A|^{\gamma' + 2/k} K^{C + 2},
$$
and taking $k = \lfloor 2/\gamma' \rfloor + 1$ and $\gamma'= \gamma/2$, the claim follows.

\end{proof}

\section{Further Applications} \label{sec:applications}

\begin{proof}[Proof of Theorem \ref{thm:STinverse}] Recall that Theorem \ref{thm:STinverse} is the following statement. For all $\gamma \geq 0$ there exists a constant $C=C(\gamma)$ such that for any finite $A\subset \mathbb Q$ with $|AA| \leq K|A|$ and any finite set $L$ of lines in the plane, $I(P,L)\leq 3 |P| + |A|^{\gamma}K^C|L|$, where $P= A \times A$.

First of all, observe that horizontal and vertical lines contribute a total of at most $2|P|$. This is because each point $p \in P$ can belong to at most one horizontal and one vertical line. Similarly, lines through the origin contribute at most $|P| +|L|$ incidences, since each point aside from the origin belongs to at most one such line, and the origin itself may contribute $|L|$ incidences.

It remains to bound incidences with lines of the form $y=mx+c$, with $m,c \neq 0$. Let $l_{m,c}$ denote the line with equation $y=mx+c$. Note that, if $m \notin \mathbb Q$ then $l_{m,c}$ contains at most one point from $P$. Indeed, suppose $l_{m,c}$ contains two distinct points $(x,y)$ and $(x',y')$ from $P$. In particular, since $A \subset \mathbb Q$, $x,y,x',y' \in \mathbb Q$. Then $l_{m,c}$ has direction $m= \frac{y-y'}{x-x'}$. Therefore, lines $l_{m,c}$ with irrational slope $m$ contribute at most $|L|$ incidences.

Next, suppose that $m \in \mathbb Q$ and $c \notin \mathbb Q$. Then $l_{m,c}$ does not contain any points from $P$, since if it did then we would have a solution to $y=mx+c$, but the left hand side is rational and the right hand side is irrational.

It remains to consider the case when $m,c \in \mathbb Q^*$. An application of Theorem \ref{thm:BS_almost_subgroups} implies that $|l_{m,c} \cap P| \leq K^C|A|^{\gamma}$. Therefore, these lines contribute a total of at most $|L|K^C|A|^{\gamma}$ incidences.

Adding together the contributions from these different types of lines completes the proof.

\end{proof}

\begin{proof}[Proof of Theorem \ref{thm:additivebasis}]
Recall that Theorem \ref{thm:additivebasis} states that, for any $\gamma > 0$ there exists $C(\gamma)$ such that
for an arbitrary $A \subset \mathbb{Q}$ with $|AA| = K|A|$ and $B, B' \subset \mathbb{Q}$,
$$
S := \left|\{(b, b') \in B \times B' : b + b' \in A\} \right| \leq 2|A|^\gamma K^C \min \{|B|^{1/2}|B'| + |B|, |B'|^{1/2}|B| + |B'| \}. 
$$

We will prove that 
\begin{equation}
S \leq 2|A|^\gamma K^C(|B'|^{1/2}|B| + |B'|).
\label{eq:case1}
\end{equation}
Since the roles of $B$ and $B'$ are interchangeable, \eqref{eq:case1} also implies that $S \leq 2|A|^\gamma K^C(|B|^{1/2}|B'| + |B|)$, and thus completes the proof.

Let $\gamma > 0$ and $C(\gamma)$, given by Theorem~\ref{thm:BS_almost_subgroups}, be fixed. Without loss of generality assume that $S \geq 2|B'|$ as otherwise the claimed bound is trivial. 

For each $b \in B$ define 
$$
S_b := \{ b' \in B' : b + b' \in A\},
$$
and similarly for $b' \in B'$
$$
T_{b'} := \{b \in B: b' + b \in A \}.
$$
It follows from Theorem~\ref{thm:BS_almost_subgroups} that for $b_1,b_2 \in B$ with $b_1 \neq b_2$ 
$$
|S_{b_1} \cap S_{b_2}| \leq |A|^\gamma K^C
$$
since each $x \in S_{b_1} \cap S_{b_2}$ gives a solution $(a, a') := (b_1 + x, b_2 + x)$ to
$$
a - a' = b_1 - b_2
$$
with $a, a' \in A$.

On the other hand, by double-counting and the Cauchy-Schwarz inequality,
$$
\sum_{b \in B} |S_b| + \sum_{b_1,b_2 \in B : b_1 \neq b_2} |S_{b_1} \cap S_{b_2}| =
\sum_{b' \in B'} |T_{b'}|^2 \geq |B'|^{-1}(\sum_{b' \in B'} |T_{b'}|)^2 = |B'|^{-1}S^2.
$$
Therefore,
$$
\sum_{b_1,b_2 \in B : b_1 \neq b_2} |S_{b_1} \cap S_{b_2}| \geq |B'|^{-1}S^2 - \sum_{b \in B} |S_b| = |B'|^{-1}S^2 - S \geq \frac{1}{2} |B'|^{-1}S^2
$$
by our assumption.

The left-hand side is at most $ |B|^2|A|^\gamma K^C$, and so
$$
	S \leq (2|A|^\gamma K^C)^{1/2} |C|^{1/2}|B'|,
$$
which completes the proof.
%After redifining $\gamma$ and $C$ the claim follows.

\end{proof}

\begin{proof}[Proof of Theorem \ref{thm:proddiff}]
Recall that Theorem \ref{thm:proddiff} states that for all $b$ there exists $k$ such that for all $A,B \subset \mathbb Q$ with $|B| \geq 2$, $|(A+B)^{k}| \geq |A|^b$.

Since $|B| \geq 2$, there exist two distinct elements $b_1, b_2 \in B$. Apply Theorem \ref{thm:mainmain} to conclude that for all $b$ there exists $k=k(b)$ with
\[ |(A+B)^k| \geq  \max \{|(A+b_1)^k|, |((A+b_1)+(b_2-b_1))^{k} |\} \geq |A|^b.\]

\end{proof}

\section*{Acknowledgements}
Oliver Roche-Newton was partially supported by the Austrian Science Fund FWF Project P 30405-N32. Dmitrii Zhelezov was supported by the Knut and Alice Wallenberg Foundation Program for Mathematics 2017. 

We thank Brendan Murphy, Cosmin Pohoata, Imre Ruzsa and Endre Szemer\'edi for helpful conversations.

%\AtEndEnvironment{thebibliography}

\bibliographystyle{plain}

\begin{thebibliography}{99}

	\bibitem{amoroso2009small} F. Amoroso and E. Viada, `Small points on subvarieties of a torus', \textit{Duke Math. J.} 150(3) (2009), 407--442.
    
    \bibitem{BRNZ} A. Balog, O. Roche-Newton and D. Zhelezov, `Expanders with superquadratic growth', \textit{Electron. J. Combin.} 24 (2017), no. 3, Paper 3.14, 17 pp.
	
	\bibitem{beukers1996equation} F. Beukers and H. P. Schlickewei, `The equation $x+ y= 1$ in finitely generated groups', \textit{Acta Arith.} 78(2) (1996), 186--199.

\bibitem{BC} J. Bourgain and M.-C. Chang, `On the size of k-fold sum and product sets of integers', \textit{J. Amer. Math. Soc.} 17, no. 2, (2004), 473-497.

\bibitem{bourgain2009sum} J. Bourgain and M.-C. Chang, `Sum-product theorems in algebraic number fields' \textit{J. Anal. Math.} 109 (2009), 253-277.

\bibitem{C} M.-C. Chang, `The Erd\H{o}s-Szemerédi problem on sum set and product set', \textit{Ann. of Math. (2)} 157, no. 3, (2003), 939-957.

\bibitem{ES} P. Erd\H{o}s and E. Szemer\'{e}di, `On sums and products of integers', \textit{Studies in pure mathematics}, Birkh\"{a}user, Basel, (1983), 213-218.

\bibitem{evertse2002linear} J.-H. Evertse, H. P. Schlickewei, and W.M. Schmidt, `Linear equations in variables which lie in a multiplicative group', \textit{Ann. of Math.} 155(3) (2002), 807--836.

\bibitem{GS} M. Garaev and C.-Y. Shen, `On the size of the set $A(A+1)$', \textit{Math. Z.} 265, no. 1, (2010), 125-132.

\bibitem{HRNZ} B. Hanson, O. Roche-Newton and D. Zhelezov, 
`On iterated product sets with shifts', \textit{Mathematika} 65 (2019), no. 4, 831-850. 


\bibitem{JRN} T. G. F. Jones and O. Roche-Newton, `Improved bounds on the set $A(A+1)$', \textit{J. Combin. Theory Ser. A} 120, no. 3, (2013), 515-526.

\bibitem{KS} S. V. Konyagin and I. D. Shkredov, `On sum sets of sets, having small product set', \textit{Proc. Steklov Inst. Math.} 290 (2015), 288-299.

\bibitem{KS2} S. V. Konyagin and I. D. Shkredov, `New results on sums and products in $\mathbb R$', \textit{Proc. Steklov Inst. Math.} 294 (2016), 87-98.

\bibitem{MRSZ} D. Matolsci, I. Z. Ruzsa, G. Shakan and D. Zhelezov, 
`An analytic approach to cardinalities of sumsets', \textit{Forthcoming}. 

\bibitem{PZ} D. P\"{a}lv\"{o}lgyi and D. Zhelezov, 
`Query complexity and the polynomial Freiman-Ruzsa conjecture', \textit{Forthcoming}. 


\bibitem{P} G. Petridis, `New proofs of Pl\"{u}nnecke-type estimates for product sets in groups', \textit{Combinatorica} 32, no. 6, (2012), 721-733.

\bibitem{shkredov2016additive} I. D. Shkredov and D. Zhelezov, `On additive bases of sets with small product set', \textit{IMRN} 2018, no. 5, (2018), 1585-1599.

\bibitem{S} J. Solymosi, `Bounding multiplicative energy by the sumset', \textit{Adv. Math.} 222 (2009), 402-408.

\bibitem{TV} T. Tao, V. Vu. 'Additive combinatorics' \textit{Cambridge University Press} (2006).

\bibitem{Z} D. Zhelezov, 
`Bourgain-Chang's proof of the weak Erd\H{o}s-Szemer\'{e}di conjecture', \textit{arXiv:1710.09316} (2017). 



\end{thebibliography}

\end{document}